\documentclass[12pt,a4paper]{article}
\usepackage{amsmath,amssymb,amsthm}
\usepackage{amscd}
\usepackage{mathtools}
\mathtoolsset{showonlyrefs}
\usepackage[all]{xy}
\usepackage{color}
\usepackage[dvipdfmx,colorlinks,setpagesize=false]{hyperref}

\newtheorem{thm}{Theorem}[section]

\newtheorem{lem}[thm]{Lemma}
\newtheorem{cor}[thm]{Corollary}

\theoremstyle{definition}
\newtheorem{defn}[thm]{Definition}

\newtheorem{exa}[thm]{Example}
\newtheorem{rmk}[thm]{Remark}
\newtheorem{notn}[thm]{Notation}
\newtheorem{para}[thm]{}

\numberwithin{equation}{thm}

\newcommand{\RomII}{\uppercase\expandafter{\romannumeral 2}}

\newcommand{\bC}{\mathbb C}
\newcommand{\bZ}{\mathbb Z}
\newcommand{\bN}{\mathbb N}
\newcommand{\bP}{\mathbb P}
\newcommand{\bR}{\mathbb R}
\newcommand{\bV}{\mathbb V}
\newcommand{\li}{1, 2, \dots, l}
\newcommand{\sli}{\{\li\}}
\newcommand{\cA}{\mathcal A}
\newcommand{\cC}{\mathcal C}
\newcommand{\cE}{\mathcal E}
\newcommand{\cF}{\mathcal F}
\newcommand{\cH}{\mathcal H}
\newcommand{\cO}{\mathcal O}
\newcommand{\cV}{\mathcal V}
\newcommand{\ca}{complex analytic }
\newcommand{\snc}{simple normal crossing }
\DeclareMathOperator{\gr}{Gr}
\DeclareMathOperator{\id}{id}
\DeclareMathOperator{\image}{Image}
\DeclareMathOperator{\kernel}{Ker}
\DeclareMathOperator{\res}{Res}
\DeclareMathOperator{\shom}{\cH {\it om}}
\DeclareMathOperator{\fvb}{FiltBun}
\DeclareMathOperator{\fpvhs}{FPVHS}
\DeclareMathOperator{\gpfmhs}{GrPFMHS}

\begin{document}

\title{A remark on semipositivity theorems}
\author{Taro Fujisawa
\thanks{Tokyo Denki University, fujisawa@mail.dendi.ac.jp}}
\maketitle

\begin{abstract}
We propose a new class of filtered vector bundles,
which is related to variation of (mixed) Hodge structures
and give a slight generalization
of the Fujita--Zucker--Kawamata semipositivity theorem.
\end{abstract}

\section*{Introduction}

The aim of this article
is to give a remark
on semipositivity theorems, Theorem 1.8 and Theorem 4.5
in \cite{Brunebarbe},
which are generalizations of the Fujita--Zucker--Kawamata
semipositivity theorem.
(cf. \cite{FujitaKFSC},
\cite{ZuckerRTF},
\cite{KawamataCAV},
\cite{Fujino-Fujisawa},
\cite{FujinoFujisawaSaito} etc.)
In fact,
Theorem 1.8 is a corollary of Theorem 4.5 in \cite{Brunebarbe}.
However, Example \ref{exa:2} below shows that
Theorem 4.5 of \cite{Brunebarbe} is false.
In this article,
we prove another semipositivity theorem,
Theorem \ref{thm:2},
and recover Theorem 1.8 of \cite{Brunebarbe}
as its corollary.
Here the author would like to mention the article
\cite{FujinoFujisawaSPT},
in which Theorem 1.8 of \cite{Brunebarbe}
is recovered and generalized from the analytic viewpoint.
Also, \cite{BrunebarbeSPHB}
treats a generalization of Theorem 1.8 of \cite{Brunebarbe}.

In this article, we adopt
the same strategy as in \cite{KawamataCAV}, \cite{Fujino-Fujisawa}
and \cite{Brunebarbe},
which uses some properties of the degeneration
of a polarized variation of $\bR$-Hodge structures.
In this approach,
there exists the following difficulty:
Even if we start from a polarized variation of Hodge structures,
the objects which appear as its degeneration
are not necessarily variations of mixed Hodge structures
as explained in \ref{para:4} below.
Therefore it is not sufficient
to consider a polarized variation of (mixed) Hodge structures.
This is the reason
why the category $\gpfmhs(X,D)_{\bR}$ was introduced
in \cite{Brunebarbe}.
However the notion of the category $\gpfmhs(X,D)_{\bR}$
contains a problem
as shown by Example \ref{exa:1} below.
In this article,
we propose a new category
$\fpvhs(X,D)_{\bR}$
and prove the semipositivity theorem
for its object.

For the proof of the semipoistivity theorem
by using the inductive argument as above,
the key is the construction of the restriction functor.
The idea of Dr. Brunebarbe is
to use the refinement of the weight filtration.
However, taking the refinement
breaks the functoriality for the weight filtration.
In order to overcome this problem,
he uses the huge index set $\bZ^{\infty}$
as in \cite[Section 4]{Brunebarbe},
which causes other problem
as in Example \ref{exa:1}.
In this article,
we change the category as follows:
In contrast to the category $\gpfmhs(X,D)_{\bR}$
in \cite[Definition 4.1]{Brunebarbe}
whose object is a filtered vector bundle $(\cV,F)$
{\itshape equipped with} an extra data $(W, \dots)$,
an object of the category $\fpvhs(X,D)_{\bR}$
is a filtered vector bundle $(\cV, F)$
which {\itshape admits} an extra data $(W, \dots)$.
Thus the category $\fpvhs(X,D)_{\bR}$
is defined as a full subcategory of the filtered vector bundles on $X$.
This definition implies
that a morphism in the category $\fpvhs(X,D)_{\bR}$
has no constraint on the weight filtration $W$.
Thus the restriction considered in Section \ref{sec:restriction-functor}
becomes a functor.

For an object $(\cV, F, W, \dots)$ of $\gpfmhs(X,D)_{\bR}$,
a subbundle $A$ of $\gr_F\cV$
is considered
in Theorem 4.5 of \cite{Brunebarbe}.
However the assumption that $\gr^WA$ is contained
in the kernel of the Higgs filed
associated to $(\cV, F, W, \dots)$
is not preserved by restricting to the subvariety of $X$.
This phenomena violates the proof of Theorem 4.5 in \cite{Brunebarbe}.
In Theorem \ref{thm:2} below,
a quotient bundle $A$ of $\gr_F\cV$
is considered for an object $(\cV,F)$ of $\fpvhs(X,D)_{\bR}$.
Then the assumption on $A$
concerning the Higgs field associated to $(\cV,F)$
is preserved by the restriction
and the inductive argument gives us
the proof of the semipositivity theorem.

This article is organized as follows.
Section \ref{sec:preliminaries}
treats preliminary facts about filtrations.
We introduce the notion of a refinement of a filtration
and prove several lemmas and corollaries concerning it.
These play important role
in Sections \ref{sec:restriction-functor}
and \ref{sec:semip-theor},
although these are of technical nature.
In the first half of Section \ref{sec:category-fpvhsx-d_br},
the ambiguities in \cite[Section 5]{Fujino-Fujisawa}
are fixed at this occasion.
Next, the category $\fpvhs(X,D)_{\bR}$ is defined
for a log pair $(X,D)$.
In Section \ref{sec:restriction-functor},
the restriction functor to a stratum of the boundary $D$
is constructed.
Then the main theorem of this article,
Theorem \ref{thm:2}, is proved in Section \ref{sec:semip-theor}.

The author would like to express his gratitude
to Dr. Yohan Brunebarbe.
The construction of the restriction functor
in Section \ref{sec:restriction-functor}
is essentially the same as his original idea.
The author learned it
from the discussion with himself.
The author would like to thank Professor Osamu Fujino
for his helpful advice and encouragement.

\section{Preliminaries}
\label{sec:preliminaries}

\begin{para}
In this section,
we collect elementary facts
concerning filtrations.
Throughout this section,
$X$ denotes a \ca space.
\end{para}

\begin{para}
Let $\cF$ and $\cV$ be coherent $\cO_X$-modules.
For a finite decreasing filtration $F$ on $\cV$,
the filtration $F$ on $\cF \otimes \cV$
is defined by
\begin{equation}
\label{eq:1}
F^p(\cF \otimes \cV)
=\image(\cF \otimes F^p\cV \longrightarrow \cF \otimes \cV)
\end{equation}
for every $p$.
Then there exist the canonical morphisms
\begin{equation}
\label{eq:10}
\cF \otimes \gr_F^p\cV
\longrightarrow
\gr_F^p(\cF \otimes \cV)
\end{equation}
for all $p$,
which is an isomorphism if $\gr_F^p\cV$ is locally free for all $p$.

Let $G$ be another finite decreasing filtration on $\cV$.
Then $G$ induces
the filtrations $G$ on $\gr_F^p\cV$ and on $\cF \otimes \cV$,
which induce the filtrations $G$ on $\cF \otimes \gr_F^p\cV$
and on $\gr_F^p(\cF \otimes \cV)$ respectively.
The morphism \eqref{eq:10}
preserves the filtration $G$ on the both sides.
Similarly, there exists the canonical morphism
\begin{equation}
\label{eq:17}
\cF \otimes \gr_G^q\cV
\longrightarrow
\gr_G^q(\cF \otimes \cV)
\end{equation}
preserving the filtration $F$ on the both sides.
Moreover, there exist the canonical morphisms
\begin{equation}
\label{eq:18}
\begin{split}
&\cF \otimes \gr_F^p\gr_G^q\cV
\longrightarrow
\gr_F^p(\cF \otimes \gr_G^q\cV), \\
&\cF \otimes \gr_G^q\gr_F^p\cV
\longrightarrow
\gr_G^q(\cF \otimes \gr_F^p\cV), \\
&\gr_F^p(\cF \otimes \gr_G^q\cV)
\longrightarrow
\gr_F^p\gr_G^q(\cF \otimes \cV), \\
&\gr_G^q(\cF \otimes \gr_F^p\cV)
\longrightarrow
\gr_G^q\gr_F^p(\cF \otimes \cV)
\end{split}
\end{equation}
for any $p,q$.
\end{para}

\begin{lem}
\label{lem:5}
The diagram
\begin{equation}
\begin{CD}
\cF \otimes \gr_G^q\gr_F^p\cV
@>>>
\gr_G^q(\cF \otimes \gr_F^p\cV)
@>>>
\gr_G^q\gr_F^p(\cF \otimes \cV) \\
@VVV @. @VVV \\
\cF \otimes \gr_F^p\gr_G^q\cV
@>>>
\gr_F^p(\cF \otimes \gr_G^q\cV)
@>>>
\gr_F^p\gr_G^q(\cF \otimes \cV)
\end{CD}
\end{equation}
is commutative for all $p,q$,
where the horizontal arrows are the ones in \eqref{eq:18},
the left vertical arrow is the one induced from
the canonical morphism for two filtrations
and the right vertical arrow is the canonical morphism
for two filtrations.
\end{lem}
\begin{proof}
Easy by definition.
\end{proof}

\begin{lem}
\label{lem:10}
Assume that
$\gr_F^p\gr_G^q\cV \simeq \gr_G^q\gr_F^p\cV$
is a locally free $\cO_X$-module of finite rank
for all $p,q$.
Then the morphism
\eqref{eq:10}
is an isomorphism for all $p$,
under which the filtrations $G$ on the both sides
are identified.
Similarly,
the morphism \eqref{eq:17}
is an isomorphism for all $q$,
under which the filtrations $F$ on the both sides
are identified.
Moreover, all the morphisms
in \eqref{eq:18}
are isomorphisms for all $p,q$.
\end{lem}
\begin{proof}
See \cite[Lemma 2.7]{FujisawaLHSSV2}.
\end{proof}

\begin{para}
\label{para:6}
Let $\cV$ be a coherent $\cO_X$-module.
The $\cO_X$-dual of $\cV$ is denoted by $\cV^{\ast}$,
that is,
$\cV^{\ast}=\shom_{\cO_X}(\cV, \cO_X)$.
If a finite decreasing filtration $F$ on $\cV$ is given,
a finite decreasing filtration $F$ on $\cV^{\ast}$ is defined by
\begin{equation}
\label{eq:25}
F^p\cV^{\ast}
=\{f \in \cV^{\ast}
\mid
f(F^{-p+1}\cV)=0\}
\simeq
(\cV/F^{-p+1}\cV)^{\ast}
\end{equation}
for every $p$.
Then a section $f \in F^p\cV^{\ast}$
defines a morphism from $\gr_F^{-p}\cV$ to $\cO_X$.
Thus we obtain the canonical morphism
\begin{equation}
F^p\cV^{\ast}
\longrightarrow
(\gr_F^{-p}\cV)^{\ast},
\end{equation}
which induces the canonical morphism
\begin{equation}
\label{eq:22}
\gr_F^p\cV^{\ast}
\longrightarrow
(\gr_F^{-p}\cV)^{\ast}
\end{equation}
for all $p$.
Let $G$ be another finite decreasing filtration on $\cV$.
Then we have the canonical morphism
\begin{equation}
\label{eq:21}
\gr_G^q\cV^{\ast}
\longrightarrow
(\gr_G^{-q}\cV)^{\ast}
\end{equation}
for all $q$.
Because we have
\begin{equation}
F^p\cV^{\ast} \cap G^q\cV^{\ast}
=\{f \in \cV^{\ast}
\mid
f(F^{-p+1}\cV)=0 \text{ and } f(G^{-q+1}\cV)=0\},
\end{equation}
the morphism \eqref{eq:21}
sends $F^p\gr_G^q\cV^{\ast}$
to $F^p(\gr_G^{-q}\cV)^{\ast}$
for all $p$.
\end{para}

\begin{lem}
\label{lem:6}
If $\gr_F^p\cV$ is a locally free $\cO_X$-module of finite rank for all $p$,
then the morphism \eqref{eq:22}
is an isomorphism for all $p$.
If $\gr_F^p\gr_G^q\cV$ is
a locally free $\cO_X$-module of finite rank for all $p,q$,
then the filtrations $F$ on $\gr_G^q\cV^{\ast}$
and on $(\gr_G^{-q}\cV)^{\ast}$ coincide
under the identification \eqref{eq:21}.
In particular, we have the canonical isomorphisms
\begin{equation}
\gr_F^p\gr_G^q\cV^{\ast}
\overset{\simeq}{\longrightarrow}
\gr_F^p(\gr_G^{-q}\cV)^{\ast}
\overset{\simeq}{\longrightarrow}
(\gr_F^{-p}\gr_G^{-q}\cV)^{\ast}
\end{equation}
for all $p,q$.
\end{lem}
\begin{proof}
Since we can easily prove the first assertions,
we prove the second assertion here.
By the assumption that $\gr_F^p\gr_G^q\cV$ is locally free of finite rank
for all $p,q$,
we may assume that there exists a direct sum decomposition
\begin{equation}
\cV=\bigoplus_{p,q}\cV^{p,q}
\end{equation}
satisfying the properties
\begin{equation}
F^p\cV=\bigoplus_{p' \ge p}\cV^{p',q}, \quad
G^q\cV=\bigoplus_{q' \ge q}\cV^{p,q'}
\end{equation}
as in the proof of Lemma 2.7 of \cite{FujisawaLHSSV2}.
Then the conclusion is trivial.
\end{proof}

Next, we define the notion of a refinement of a filtration
and prove several elementary properties of it.

\begin{defn}
Let $\cA$ be an abelian category,
$V$ an object of $\cA$
and $W$ a finite increasing filtration on $V$.
A refinement of $W$ is a pair $(M, \varphi)$
consisting of a finite increasing filtration $M$
and a strictly increasing map
$\varphi: \bZ \longrightarrow \bZ$
satisfying
$W_mV=M_{\varphi(m)}V$ for all $m$.
Sometimes we say $M$ is a refinement of $W$
if it is not necessary to specify the map $\varphi$.
\end{defn}

\begin{defn}
\label{defn:1}
Let $\varphi: \bZ \longrightarrow \bZ$ be
a strictly increasing map.
Then for any $k \in \bZ$,
there exists the unique $m \in \bZ$
such that $\varphi(m-1) < k \le \varphi(m)$.
This integer $m$ is denoted by $m_{\varphi}(k)$,
or simply $m(k)$ if there is no danger of confusion.
By definition
$\varphi(m(k)-1) < k \le \varphi(m(k))$
for every $k \in \bZ$.
\end{defn}

\begin{lem}
\label{lem:4}
Let $V$ and $W$ be as above
and $(M, \varphi)$ a refinement of $W$.
Then we have
\begin{equation}
W_{m(k)-1}V=M_{\varphi(m(k)-1)}V
\subset
M_{k-1}V
\subset
M_kV
\subset
M_{\varphi(m(k))}V=W_{m(k)}V
\end{equation}
for all $k \in \bZ$.
Therefore we have the canonical surjection
\begin{equation}
M_kV \longrightarrow M_k\gr_{m(k)}^WV
\end{equation}
which induces an isomorphism
\begin{equation}
\label{eq:11}
\gr_k^MV
\overset{\simeq}{\longrightarrow}
\gr_k^M\gr_{m(k)}^WV
\end{equation}
for all $k \in \bZ$.
\end{lem}

\begin{lem}
\label{lem:2}
Let $V$ and $W$ be as above,
$M$ a refinement of $W$
and $F$ a finite decreasing filtration on $V$.
Under the isomorphism
\eqref{eq:11},
\begin{equation}
F^p\gr_k^MV
\simeq
F^p\gr_k^M\gr_{m(k)}^WV
\end{equation}
for all $k,p$.
\end{lem}
\begin{proof}
We have the canonical surjection
\begin{equation}
\label{eq:12}
F^p\gr_{m(k)}^WV \cap M_k\gr_{m(k)}^WV
\longrightarrow
F^p\gr_k^M\gr_{m(k)}^WV
\end{equation}
by definition.
If $x \in F^pV \cap W_{m(k)}V$
and $y \in M_kV \cap W_{m(k)}V=M_kV$
define the same element in $\gr_{m(k)}^WV$,
then $x-y=z \in W_{m(k)-1}V \subset M_{k-1}V$.
Hence $x=y+z \in F^pV \cap M_kV$.
Thus the canonical morphism
\begin{equation}
\label{eq:13}
F^pV \cap M_kV
\longrightarrow
F^p\gr_{m(k)}^WV \cap M_k\gr_{m(k)}^WV
\end{equation}
is surjective.
Combining the surjectivity of the morphism \eqref{eq:12},
the canonical morphism
\begin{equation}
F^pV \cap M_kV
\longrightarrow
F^p\gr_k^M\gr_{m(k)}^WV
\end{equation}
is surjective.
Thus we obtain the conclusion.
\end{proof}

\begin{cor}
In the situation above,
we have the canonical isomorphisms
\begin{equation}
\label{eq:15}
\gr_F^p\gr_k^MV
\overset{\simeq}{\longrightarrow}
\gr_F^p\gr_k^M\gr_{m(k)}^WV
\end{equation}
for all $k,p$.
\end{cor}

\begin{lem}
\label{lem:3}
We have
\begin{equation}
M_k\gr_F^p\gr_{m(k)}^WV
\simeq
M_k\gr_{m(k)}^W\gr_F^pV
\end{equation}
under the canonical identification
\begin{equation}
\label{eq:14}
\gr_F^p\gr_{m(k)}^WV
\simeq
\gr_{m(k)}^W\gr_F^pV
\end{equation}
for all $k,p$.
\end{lem}
\begin{proof}
On $\gr_F^pV$,
we have $M_k\gr_F^pV \subset W_{m(k)}\gr_F^pV$ by definition.
Therefore the canonical morphism
\begin{equation}
M_kV \cap F^pV
\longrightarrow
M_k\gr_{m(k)}^W\gr_F^pV
\end{equation}
is surjective.
On the other hand,
the canonical morphism
\begin{equation}
M_k\gr_{m(k)}^WV \cap F^p\gr_{m(k)}^WV
\longrightarrow
M_k\gr_F^p\gr_{m(k)}^WV
\end{equation}
is surjective by definition.
Combining the surjectivity of the morphism \eqref{eq:13},
the canonical morphism
\begin{equation}
M_kV \cap F^pV
\longrightarrow
M_k\gr_F^p\gr_{m(k)}^WV
\end{equation}
is surjective.
Thus we obtain the conclusion.
\end{proof}

\begin{cor}
The canonical isomorphism
\eqref{eq:14}
induces the isomorphism
\begin{equation}
\label{eq:16}
\gr_k^M\gr_F^p\gr_{m(k)}^WV
\simeq
\gr_k^M\gr_{m(k)}^W\gr_F^pV
\end{equation}
for all $m,p$.
\end{cor}

\begin{cor}
\label{cor:1}
In the situation above,
the diagram
\begin{equation}
\xymatrix{
\gr_F^p\gr_k^MV \ar[r]_-{(1)}^-{\simeq} \ar[dd]_-{\simeq}^-{(2)}
& \gr_F^p\gr_k^M\gr_{m(k)}^WV \ar[d]_-{(3)}^-{\simeq} \\
& \gr_k^M\gr_F^p\gr_{m(k)}^WV \ar[d]^-{\simeq}_-{(4)} \\
\gr_k^M\gr_F^pV \ar[r]^-{\simeq}_-{(5)} & \gr_k^M\gr_{m(k)}^W\gr_F^pV
}
\end{equation}
is commutative for all $k,p$,
where $(1)$ is the isomorphism \eqref{eq:15},
$(2)$ and $(3)$ are the canonical isomorphisms
switching two filtrations $F, M$
on $V$ and $\gr_{m(k)}^WV$ respectively,
$(4)$ is the isomorphism \eqref{eq:16}
and $(5)$ is the isomorphism \eqref{eq:11}
for $\gr_F^pV$.
\end{cor}
\begin{proof}
By the proof of Lemmas \ref{lem:2} and \ref{lem:3},
we can easily see that
all the isomorphisms $(1)$--$(5)$
are induced by the surjection from $F^pV \cap M_kV$.
Thus the conclusion is trivial.
\end{proof}

\section{The category $\fpvhs(X,D)_{\bR}$}
\label{sec:category-fpvhsx-d_br}

\begin{notn}
\label{notn:1}
Let $(X, D)$ be a log pair,
that is, $X$ a smooth complex variety and $D$ a \snc divisor on $X$.
The irreducible decomposition of $D$ is given by $D=\sum_{i \in I}D_i$.
We set
\begin{equation}
D(J)=\bigcap_{i \in J}D_i, \quad
D_J=\sum_{i \in J}D_i
\end{equation}
for $J \subset I$.
For the case of $J= \emptyset$,
$D(\emptyset)=X$ and $D_{\emptyset}=0$ by definition.
Moreover, we use the notation
\begin{equation}
D(J)^{\ast}=D(J) \setminus D(J) \cap D_{I \setminus J}
\end{equation}
for $J \subset I$.
For $J=\emptyset$, $X^{\ast}=X \setminus D$ by definition.
Then $(D(J), D(J) \cap D_{I \setminus J})$ is a log pair again.
For a log pair $(X,D)$,
\begin{equation}
\omega^p_X=\Omega^p_X(\log D)
\end{equation}
for every $p$
as in \cite{KazuyaKato} for short,
if there is no danger of confusion.
Thus
\begin{equation}
\omega^p_{D(J)}=\Omega^p_{D(J)}(\log D(J) \cap D_{I \setminus J})
\end{equation}
for a log pair $(D(J), D(J) \cap D_{I \setminus J})$
for a subset $J \subset I$.
\end{notn}

\begin{para}
Firstly, we add an explanation
to the presentation
given in \cite[Section 5]{Fujino-Fujisawa}.
In 5.8 of \cite{Fujino-Fujisawa},
the condition ($m$MH) is defined.
However, it was not precise enough
because the real structure was not mentioned.
Here the precise statements are given,
which are sufficient
for the argument in \cite[Section 5]{Fujino-Fujisawa}.
\end{para}

\begin{para}
\label{para:1}
Let $(X, D)$ be a log pair.
We assume that the following data is given:
\begin{itemize}
\item
a locally free $\cO_X$-module of finite rank $\cV$,
\item
an integrable log connection $\nabla$ on $\cV$
with the nilpotent residues,
\item
an $\bR$-local subsystem $\bV$ of $\kernel(\nabla)|_{X^{\ast}}$
such that $\bC \otimes \bV=\kernel(\nabla)|_{X^{\ast}}$.
\end{itemize}
The residue morphism along $D_i$
is denoted by $\res_{D_i}(\nabla)$ for $i \in I$.
The morphism $\res_{D_i}(\nabla)$ is nilpotent by definition.
The restriction of $\res_{D_i}(\nabla)$ to
$\cO_{D(J)} \otimes \cV$ is denoted by
$\res_{D_i}(\nabla)|_{D(J)}$ for $i \in J$.
We set
\begin{equation}
\res_{K,D(J)}(\nabla)=\sum_{i \in K}\res_{D_i}(\nabla)|_{D(J)}
\end{equation}
for $K \subset J$.
The monodromy weight filtration for $\res_{K,D(J)}(\nabla)$
on $\cO_{D(J)} \otimes \cV$
is denoted by $W(K)$
for $K \subset J \subset I$
as in \cite{Fujino-Fujisawa}.
\end{para}

\begin{para}
\label{para:2}
Now we treat the local situation.
Namely, let us assume that $X$ is the polydisc $\Delta^n$
with the coordinates $(t_1, t_2, \dots, t_n)$
and $D=\{t_1t_2 \cdots t_l=0\}$ for some $l$ with $1 \le l \le n$.
Then $I=\sli$.

Let $(\cV, \nabla, \bV)$ be
as in \ref{para:1}.
By the local description in
\cite{KatzDW},
there exist
\begin{itemize}
\item
a finite dimensional $\bR$-vector space $V$,
\item
nilpotent endomorphisms
$N_1, N_2, \dots, N_l$ of $V$
satisfying the property $N_iN_j=N_jN_i$ for all $i,j \in \sli$
and
\item
an isomorphism of $\cO_X$-modules
$\varphi: \cO_X \otimes V \longrightarrow \cV$
\end{itemize}
such that the following properties hold:
\begin{itemize}
\item
The integrable log connection
$\varphi^{\ast}\nabla:
\cO_X \otimes V \longrightarrow \omega^1_X \otimes V$
is given by
\begin{equation}
\label{eq:2}
(\varphi^{\ast}\nabla)(f \otimes v)
=df \otimes v
-(2\pi\sqrt{-1})^{-1}f\sum_{i=1}^{l}\frac{dt_i}{t_i} \otimes N_i(v).
\end{equation}
\item
The $\bR$-local system $\varphi^{-1}\bV$ is the image
of the (multi-valued) $\bR$-morphism
\begin{equation}
\label{eq:5}
\exp((2\pi\sqrt{-1})^{-1}\sum_{i=1}^{l}(\log t_i)N_i):
V \longrightarrow \cO_X \otimes V.
\end{equation}
\end{itemize}
Thus we may assume that
the data $(\cV, \nabla, \bV)$
is given by the following:
\begin{itemize}
\item
$\cV=\cO_X \otimes V$,
\item
the integrable log connection $\nabla$ is given by
the right hand side of
\eqref{eq:2},
\item
the $\bR$-local system $\bV$ is the image of the morphism
given in \eqref{eq:5}.
\end{itemize}
Then we have
\begin{equation}
\res_{D_i}(\nabla)=-(2\pi\sqrt{-1})^{-1}(\id \otimes N_i):
\cO_{D_i} \otimes V
\longrightarrow
\cO_{D_i} \otimes V
\end{equation}
for every $i$.
Moreover
\begin{equation}
\res_{K,D(J)}(\nabla)
=-(2\pi\sqrt{-1})^{-1}(\id \otimes N_K):
\cO_{D(J)} \otimes V
\longrightarrow
\cO_{D(J)} \otimes V
\end{equation}
where $N_K=\sum_{i \in K}N_i$ for $K \subset J \subset I$.
Now the monodromy weight filtration of $N_K$ on $V$
is denoted by $W(K)$ for $K \subset I$.
Then we have
\begin{equation}
W(K)_m(\cO_{D(J)} \otimes \cV)=\cO_{D(J)} \otimes W(K)_mV
\end{equation}
for $K \subset J \subset I$ and for all $m$.
Therefore
\begin{equation}
\gr_m^{W(K)}(\cO_{D(J)} \otimes \cV)
\simeq
\cO_{D(J)} \otimes \gr_m^{W(K)}V
\end{equation}
is a free $\cO_{D(J)}$-module
of finite rank for every $m$.
\end{para}

\begin{para}
In the situation above,
let $J \subset I=\sli$ be a subset.
On the free $\cO_{D(J)}$-module
$\cO_{D(J)} \otimes \cV \simeq \cO_{D(J)} \otimes V$,
an integrable log connection $\nabla(J)$ is defined by
\begin{equation}
\label{eq:7}
\nabla(J)(f \otimes v)
=df \otimes v
-(2\pi\sqrt{-1})^{-1}f\sum_{i \in I \setminus J}
\frac{dt_i}{t_i} \otimes N_i(v)
\end{equation}
as in \eqref{eq:2},
and an $\bR$-local subsystem $\bV(J)$
of $\kernel(\nabla(J))|_{D(J)^{\ast}}$
is defined as the image of the morphism
\begin{equation}
\label{eq:8}
\exp((2\pi\sqrt{-1})^{-1}\sum_{i \in I \setminus J}(\log t_i)N_i):
V \longrightarrow \cO_{D(J)} \otimes V,
\end{equation}
as in \eqref{eq:5}.
Then we have
\begin{equation}
\bC \otimes \bV(J)
=\kernel(\nabla(J))|_{D(J)^{\ast}}
\end{equation}
as before.
For any $K \subset J$,
a finite increasing filtration $W(K)$ on $\bV(J)$
is obtained as the image of $W(K)$ on $V$
by the morphism \eqref{eq:8}.
Then $W(K)_m\bV(J)$ and $\gr_m^{W(K)}\bV(J)$
are $\bR$-local systems on $D(J)^{\ast}$ for all $m$.
The inclusion
$\kernel(\nabla(J)) \longrightarrow \cO_{D(J)} \otimes \cV$
induces the isomorphism
\begin{equation}
(\cO_{D(J)} \otimes \kernel(\nabla(J)))|_{D(J)^{\ast}}
\simeq
(\cO_{D(J)} \otimes \cV)|_{D(J)^{\ast}}
\end{equation}
under which we have the identification
\begin{equation}
\cO_{D(J)^{\ast}} \otimes W(K)_m\bV(J)
\simeq
W(K)_m(\cO_{D(J)} \otimes \cV)|_{D(J)^{\ast}}
\end{equation}
for all $m$.
Therefore we have
\begin{equation}
\label{eq:3}
\bC \otimes \bV(J)_x
\simeq
\cV(x)
\end{equation}
for any $x \in D(J)^{\ast}$,
under which we have the identification
\begin{equation}
\bC \otimes W(K)_m\bV(J)_x
\simeq
W(K)_m\cV(x)
\end{equation}
for all $m$.
\end{para}

Now the following is the precise form of the condition
($m$MH) instead of the one in 5.8 of \cite{Fujino-Fujisawa}.
The point is that this condition must be considered
in the local situation as in \ref{para:2}.

\begin{defn}
Let $(X,D)$ and $(\cV, \nabla, \bV)$ be
as in \ref{para:2}
and $F$ a finite decreasing filtration on $\cV$
such that $\gr_F^p\cV$ is $\cO_X$-coherent
for all $p$.
Note that $\gr_F^p\cV$ is not assumed
to be a locally free $\cO_X$-module.
Then we say that $(\cV, \nabla, \bV, F)$ satisfies
the condition ($m$MH) for $m \in \bZ$,
if the data
\begin{equation}
((\bV(J)_x, W(J)[m]),
(\cV(x), W(J)[m], F))
\end{equation}
is an $\bR$-mixed Hodge structure for any $J \subset I$
and for any $x \in D(J)^{\ast}$,
where the isomorphism $\bC \otimes \bV(J)_x \simeq \cV(x)$
is given by \eqref{eq:3}.
Sometimes we say that $F$,
instead of $(\cV, \nabla, \bV, F)$, satisfies
the condition $(m{\rm MH})$
if there is no danger of confusion.
\end{defn}

\begin{thm}
\label{thm:1}
Let $(X,D)$ and $(\cV, \nabla, \bV)$
be as in {\rm \ref{para:2}},
$U$ an open subset of $X \setminus D$
such that $X \setminus U$ is a nowhere dense closed analytic subset of $X$
and $F$ a finite decreasing filtration on $\cV|_U$.
If the data
\begin{equation}
(\bV|_{U}, (\cV|_U, F))
\end{equation}
is a polarizable variation of $\bR$-Hodge structures of weight $m$ on $U$,
then there exists
a unique finite decreasing filtration $\widetilde{F}$ on $\cV$
such that the following conditions are satisfied:
\begin{itemize}
\item
$\widetilde{F}^p\cV|_U=F^p\cV|_U$ for all $p$
\item
$\gr_{\widetilde{F}}^p\cV$
is a locally free $\cO_X$-module of finite rank for all $p$
\item
$(\cV, \widetilde{F})$ satisfies the condition $(m{\rm MH})$.
\end{itemize}
\end{thm}
\begin{proof}
See \cite{Schmid}.
\end{proof}

\begin{lem}
Let $(X,D)$, $(\cV, \nabla, \bV)$ and $U$
be as in Theorem {\rm \ref{thm:1}}
and $F$ a finite decreasing filtration on $\cV$
such that $\gr_F^p\cV$ is $\cO_X$-coherent for all $p$.
Moreover we assume that
$(\bV, (\cV, F))|_U$ is
a polarizable variation of $\bR$-Hodge structures on $U$.
Then $\gr_F^p\cV$ is $\cO_X$-locally free for all $p$
if and only if $(\cV, F)$ satisfies the condition $(m{\rm MH})$.
\end{lem}
\begin{proof}
The proof is almost the same
as the one of Lemma 5.10 in \cite{Fujino-Fujisawa}.
Here we remark some missing point there.
Let $\widetilde{F}$ be the filtration on $\cV$
in Theorem \ref{thm:1}.
Lemma 5.1 of \cite{Fujino-Fujisawa}
implies the inclusion $F^p\cV \subset \widetilde{F}^p\cV$
for all $p$.
If $F$ satisfies the condition ($m${\rm MH}),
then $F$ and $\widetilde{F}$ induces the same filtration on $\cV(x)$
for any $x \in X$
as shown in the proof of Lemma 5.10 in \cite{Fujino-Fujisawa}.
Then we can see the coincidence of $F$ and $\widetilde{F}$ as follows:
The coherent $\cO_X$-module $\widetilde{F}^p\cV/F^p\cV$
is denoted by $\cC^p$ for a while.
From the fact
that $\cV/\widetilde{F}^p\cV$ is locally free of finite rank,
we have
${\frak m}_x\cV_x \cap \widetilde{F}^p\cV_x={\frak m}_x\widetilde{F}^p\cV_x$
for every $x \in X$ and for every $p$.
Then we have an commutative diagram
\begin{equation}
\begin{CD}
0 @>>> {\frak m}_x\cV_x \cap F^p\cV_x
  @>>> F^p\cV_x
  @>>> F^p(\cV(x)) @>>> 0 \\
@. @VVV @VVV @| \\
0 @>>> {\frak m}_x\widetilde{F}^p\cV_x
  @>>> \widetilde{F}^p\cV_x
  @>>> \widetilde{F}^p(\cV(x)) @>>> 0
\end{CD}
\end{equation}
with exact rows.
The easy diagram chasing shows the equality
$\cC^p_x={\frak m}_x\cC^p_x$.
Thus $\cC_x^p=0$ by Nakayama's lemma  for every $x \in X$.
\end{proof}

Now we define a category $\fpvhs(X,D)_{\bR}$,
which is a replacement of $\gpfmhs(X,D)_{\bR}$
given in \cite[Section 4x1]{Brunebarbe}.

\begin{defn}
\label{defn:2}
Let $X$ be a \ca space.
The category of the filtered vector bundles is denoted by
$\fvb(X)$ as in \cite{Brunebarbe}.
More precisely, the category $\fvb(X)$ is defined as follows:
An object of $\fvb(X)$ is a pair $(\cV, F)$,
where $\cV$ is a locally free $\cO_X$-module of finite rank
and $F$ is a finite decreasing filtration on $\cV$.
For two objects $(\cV_1, F)$ and $(\cV_2, F)$ of $\fvb(X)$,
a morphism $(\cV_1, F) \longrightarrow (\cV_2, F)$ in $\fvb(X)$
is a morphism of $\cO_X$-modules $\cV_1 \longrightarrow \cV_2$
preserving the filtration $F$ on $\cV_1$ and $\cV_2$.
\end{defn}

\begin{defn}
Let $(X, D)$ be a log pair.
For an object $(\cV,F)$ of $\fvb(X)$,
we consider the data
\begin{equation}
(W, \{\nabla_m\}_{m \in \bZ}, \{\bV_m\}_{m \in \bZ}, \{S_m\}_{m \in \bZ})
\end{equation}
consisting of
\begin{itemize}
\item
a finite increasing filtration $W$
on $\cV$
such that $\gr_F^p\gr_m^W\cV$ is locally free of finite rank
for all $m,p \in \bZ$,
\item
a nilpotent integrable log connection
$\nabla_m$ on $\gr_m^W\cV$
satisfying
\begin{equation}
\nabla_m(F^p\gr_m^W\cV) \subset F^{p-1}\gr_m^W\cV
\end{equation}
for all $p$
(the Griffiths transversality)
and for every $m \in \bZ$,
\item
an $\bR$-local subsystem $\bV_m$ of
$\kernel(\nabla_m)|_{X \setminus D}$
on $X \setminus D$
such that
\begin{equation}
\bC \otimes \bV_m=\kernel(\nabla_m)|_{X \setminus D}
\end{equation}
for every $m$,
\item
a morphism of $\cO_X$-modules
$S_m:
\gr_m^W\cV \otimes \gr_m^W\cV
\longrightarrow
\cO_X$
satisfying the equality
\begin{equation}
\label{eq:6}
S_m \cdot (\nabla_m \otimes \id)+S_m \cdot (\id \otimes \nabla_m)
=d \cdot S_m,
\end{equation}
for every $m \in \bZ$,
where $d$ is the usual differential regarded as a morphism
$\cO_X \longrightarrow \omega^1_X$.
\end{itemize}
Here an integrable log connection
is said to be nilpotent for short,
if all of its residues are nilpotent.
We remark that $\nabla_m$
are assumed to be nilpotent in the data above.
It is equivalent
that $\bV_m$ is assumed to be of unipotent local monodromies.

The data
$(W, \{\nabla_m\}_{m \in \bZ}, \{\bV_m\}_{m \in \bZ}, \{S_m\}_{m \in \bZ})$
above is called a structure of
filtered variation of polarized $\bR$-Hodge structures on $(\cV, F)$
if the data
\begin{equation}
(\bV_m, (\gr_m^W\cV, F)|_{X \setminus D},
\nabla_m|_{X \setminus D}, S_m|_{X \setminus D})
\end{equation}
is a variation of polarized $\bR$-Hodge structures of a certain weight
on $X \setminus D$ for all $m$.
The category $\fpvhs(X,D)_{\bR}$ is defined
as a full subcategory of $\fvb(X)$
consisting of the objects
admitting a structure of
filtered variation of polarized $\bR$-Hodge structures.
By the definition above,
for two objects
$(\cV_1, F)$ and $(\cV_2, F)$ of $\fpvhs(X,D)_{\bR}$,
a morphism from $(\cV_1, F)$ to $(\cV_2, F)$
is just a morphism of $\cO_X$-modules
$\cV_1 \longrightarrow \cV_2$ preserving the filtration $F$.
Here we note that no assumption is imposed
concerning about the structures of
filtered variation of polarized $\bR$-Hodge structures
for a morphism in $\fpvhs(X,D)_{\bR}$.
\end{defn}

\begin{rmk}
An object $(\cV, F)$ of $\fpvhs(X,D)_{\bR}$
is said to be pure,
if there exists
a structure of
filtered variation of polarized $\bR$-Hodge structures
$(W, \{\nabla_m\}_{m \in \bZ}, \{\bV_m\}_{m \in \bZ}, \{S_m\}_{m \in \bZ})$
on $(\cV, F)$
such that $W_{m_0-1}\cV=0$ and $W_{m_0}\cV=\cV$
for some integer $m_0$.
\end{rmk}

\begin{rmk}
\label{rmk:1}
For an object $(\cV, F)$ of $\fpvhs(X,D)_{\bR}$,
the $\cO_X$-module
$\gr_F^p\cV$ is locally free of finite rank for all $p$.
\end{rmk}

\begin{para}
Let $(\cV, F)$ be an object of $\fpvhs(X,D)_{\bR}$
and $(W, \{\nabla_m\}, \{\bV_m\}, \{S_m\})$
a structure of filtered polarized variation of $\bR$-Hodge structure
on $(\cV,F)$.
We set
\begin{equation}
\gr_F\gr^W\cV=\bigoplus_{p,m}\gr_F^p\gr_m^W\cV
\end{equation}
which is canonically isomorphic to
\begin{equation}
\gr^W\gr_F\cV
=\bigoplus_{p, m}\gr_m^W\gr_F^p\cV
\end{equation}
as usual.
The integrable log connection $\nabla_m$ induces
a morphism of $\cO_X$-modules
\begin{equation}
\theta_{m,p}:
\gr_F^p\gr_m^W\cV
\longrightarrow
\omega^1_X
\otimes
\gr_F^{p-1}\gr_m^W\cV
\end{equation}
by the Griffiths transversality.
We set
\begin{equation}
\theta=\bigoplus_{m,p}\theta_{m,p}:
\gr_F\gr^W\cV
\longrightarrow
\omega^1_X
\otimes
\gr_F\gr^W\cV
\end{equation}
as in \cite{Brunebarbe}.
The morphisms $\theta_{m,p}$
and $\theta$ are called the Higgs fields
associated to the given
structure of filtered polarized variation of $\bR$-Hodge structure,
or simply the Higgs fields associated to $(\cV,F)$
by abuse of the language.
For any $\chi \in \Theta_X(-\log D)=\shom_{\cO_X}(\omega^1_X, \cO_X)$,
the composite
\begin{equation}
\begin{CD}
\gr_F\gr^W\cV
@>{\theta}>>
\omega_X^1 \otimes \gr_F\gr^W\cV
@>{\chi \otimes \id}>>
\gr_F\gr^W\cV
\end{CD}
\end{equation}
is denoted by $\theta(\chi)$.
Similarly, $\theta_{m,p}(\chi)$ denotes the composite
\begin{equation}
\begin{CD}
\gr_F^p\gr_m^W\cV
@>{\theta_{m,p}}>>
\omega_X^1 \otimes \gr_F^{p-1}\gr_m^W\cV
@>{\chi \otimes \id}>>
\gr_F^{p-1}\gr_m^W\cV
\end{CD}
\end{equation}
for all $m,p$
\end{para}

Now we remark about
the pull-back of an object of $\fpvhs(X,D)_{\bR}$
by a morphism of log pairs.

\begin{para}
Let $(X,D)$ and $(Y,E)$ be log pairs
and $f: (Y,E) \longrightarrow (X,D)$ a morphism of log pairs,
that is,
a morphism of complex varieties $f: Y \longrightarrow X$
with the property
$f^{-1}D \subset E$.

Let $\cV$ be a locally free $\cO_X$-module of finite rank
equipped with a finite decreasing filtration $F$.
On a locally free $\cO_Y$-module $f^{\ast}\cV$,
a finite decreasing filtration $F$ is defined by
\begin{equation}
\label{eq:24}
F^pf^{\ast}\cV=\image(f^{\ast}F^p\cV \longrightarrow f^{\ast}\cV)
\end{equation}
for all $p$.
This filtration is called
the pull-back of the filtration $F$ on $\cV$.
Thus a functor
\begin{equation}
\label{eq:30}
f^{\ast}: \fvb(X) \longrightarrow \fvb(Y)
\end{equation}
is obtained.
We have the canonical surjective morphisms
\begin{equation}
\begin{split}
&f^{\ast}F^p\cV \longrightarrow F^pf^{\ast}\cV \\
&f^{\ast}\gr_F^p\cV \longrightarrow \gr_F^pf^{\ast}\cV
\end{split}
\end{equation}
for all $p$.
If we assume that $\gr_F^p\cV$ is locally free for all $p$,
then these morphisms are isomorphisms for all $p$.

Let $(\cV,F)$ be an object of $\fpvhs(X,D)_{\bR}$.
We fix
a structure of filtered polarized variation of $\bR$-Hodge structure
\begin{equation}
(W, \{\nabla_m\}_{m \in \bZ}, \{\bV_m\}_{m \in \bZ}, \{S_m\}_{m \in \bZ})
\end{equation}
on $(\cV,F)$.
An finite increasing filtration $W$ on $f^{\ast}\cV$ is defined by
\begin{equation}
W_mf^{\ast}\cV=\image(f^{\ast}W_m\cV \longrightarrow f^{\ast}\cV)
\end{equation}
for all $m$.
The filtration $W$ on $\gr_F^p\cV$
induces the filtration $W$ on $f^{\ast}\gr_F^p\cV$
as in \eqref{eq:24}.
Similarly, the filtration $F$ on $\gr_m^W\cV$
induces the filtration $F$ on $f^{\ast}\gr_m^W\cV$.
Since $\gr_F^p\gr_m^W\cV$ is assumed to be locally free of finite rank,
we obtain the following:
\begin{itemize}
\item
The canonical morphism
\begin{equation}
\label{eq:36}
f^{\ast}\gr_F^p\cV \longrightarrow \gr_F^pf^{\ast}\cV
\end{equation}
is an isomorphism for all $p$,
under which the filtration $W$ on the both sides
are identified.
\item
The canonical morphism
\begin{equation}
\label{eq:28}
f^{\ast}\gr_m^W\cV \longrightarrow \gr_m^Wf^{\ast}\cV
\end{equation}
is an isomorphism for all $m$,
under which the filtration $F$ on the both sides
are identified.
\item
In particular, we have the isomorphisms
\begin{equation}
\label{eq:32}
\begin{split}
&f^{\ast}\gr_m^W\gr_F^p\cV
\overset{\simeq}{\longrightarrow}
\gr_m^Wf^{\ast}\gr_F^p\cV
\overset{\simeq}{\longrightarrow}
\gr_m^W\gr_F^pf^{\ast}\cV \\
&f^{\ast}\gr_F^p\gr_m^W\cV
\overset{\simeq}{\longrightarrow}
\gr_F^pf^{\ast}\gr_m^W\cV
\overset{\simeq}{\longrightarrow}
\gr_F^p\gr_m^Wf^{\ast}\cV
\end{split}
\end{equation}
for all $m,p$.
Therefore $\gr_F^p\gr_m^Wf^{\ast}\cV$ is locally free of finite rank
for all $m,p$.
\end{itemize}
The proof of these facts are similar
to Lemma \ref{lem:10}
(see \cite[Lemma 2.7]{FujisawaLHSSV2}).
Via the identification \eqref{eq:28},
we obtain the data
\begin{equation}
f^{\ast}\nabla_m, f^{-1}\bV_m, f^{\ast}S_m
\end{equation}
on $\gr_m^Wf^{\ast}\cV$ for all $m$.
Then it is easy to see that the data
\begin{equation}
\label{eq:31}
(W, \{f^{\ast}\nabla_m\}, \{f^{-1}\bV_m\}, \{f^{\ast}S_m\})
\end{equation}
on $f^{\ast}\cV$ is
a structure of filtered polarized variation of $\bR$-Hodge structure.
Therefore $(f^{\ast}\cV, F)$ is an object of $\fpvhs(Y,E)_{\bR}$.
Thus we obtain a functor
\begin{equation}
f^{\ast}:
\fpvhs(X,D)_{\bR}
\longrightarrow
\fpvhs(Y,E)_{\bR}
\end{equation}
as the restriction of the functor \eqref{eq:30}.
By definition,
the Higgs field associated to the data
\eqref{eq:31} on $(f^{\ast}\cV, F)$
coincides with the composite
\begin{equation}
f^{\ast}\gr_F\gr^W\cV
\overset{f^{\ast}\theta}{\longrightarrow}
f^{\ast}\omega^1_X \otimes f^{\ast}\gr_F\gr^W\cV
\longrightarrow
\omega^1_Y \otimes f^{\ast}\gr_F\gr^W\cV
\end{equation}
under the identification \eqref{eq:32},
where $\theta$ denotes the Higgs field associated to $(\cV, F)$
and the second morphism is induced
from the canonical morphism
$f^{\ast}\omega^1_X \longrightarrow \omega^1_Y$.
\end{para}

For the later use,
we discuss the dual of an object of $\fpvhs(X,D)_{\bR}$.

\begin{para}
Let $(\cV,F)$ be an object of $\fpvhs(X,D)_{\bR}$.
On the dual $\cO_X$-module
$\cV^{\ast}=\shom_{\cO_X}(\cV, \cO_X)$ of $\cV$,
a finite decreasing filtration $F$ is defined
as in \eqref{eq:25}.
Now we fix
a structures of filtered polarized variation of $\bR$-Hodge structure
$(W, \{\nabla_m\}, \{\bV_m\}, \{S_m\})$
on $(\cV,F)$.
Then 
a finite increasing filtration $W$ on $\cV^{\ast}$ is defined
by a similar way to \eqref{eq:25},
that is,
\begin{equation}
W_m\cV^{\ast}
=\{f \in \cV^{\ast}
\mid
f(W_{-m-1}\cV)=0\}
\simeq
(\cV/W_{-m-1}\cV)^{\ast}
\end{equation}
for every $m$.
By Lemma \ref{lem:6},
we have the canonical isomorphisms
\begin{align}
&\gr_F^p\cV^{\ast}
\overset{\simeq}{\longrightarrow}
(\gr_F^{-p}\cV)^{\ast} \\
&\gr_m^W\cV^{\ast}
\overset{\simeq}{\longrightarrow}
(\gr_{-m}^W\cV)^{\ast}
\label{eq:26} \\
&\gr_F^p\gr_m^W\cV^{\ast}
\overset{\simeq}{\longrightarrow}
(\gr_F^{-p}\gr_{-m}^W\cV)^{\ast}
\end{align}
for all $m,p$.
In particular,
$\gr_F^p\gr_m^W\cV^{\ast}$ is locally free of finite rank
for all $m,p$.
Under the identification \eqref{eq:26},
we can easily check that the data
\begin{equation}
(W, \{\nabla_{-m}^{\ast}\}, \{\bV_{-m}^{\ast}\}, S_{-m}^{\ast})
\end{equation}
is a structure of filtered polarized variation of $\bR$-Hodge structure
on $(\cV^{\ast}, F)$,
where $\nabla_{-m}^{\ast}$ is the dual connection of $\nabla_{-m}$
on $(\gr_{-m}^W\cV)^{\ast}$,
$\bV^{\ast}_{-m}$ is the dual local system of $\bV_m$,
that is, $\bV^{\ast}_m=\shom_{\bR}(\bV_m, \bR)$,
and $S_{-m}^{\ast}$ is the dual polarization of $S_{-m}$.
This structure of filtered polarized variation of $\bR$-Hodge structure
on $(\cV^{\ast}, F)$
is called the dual structure
of the given filtered polarized variation of $\bR$-Hodge structure.
Thus $(\cV^{\ast},F)$ is an object of $\fpvhs(X,D)_{\bR}$.
The Higgs field of $(\cV^{\ast},F)$
associated to the dual structure above
is denoted by
$\theta^{\ast}=\bigoplus_{m,p}\theta_{m,p}^{\ast}$
for a while.
Since the dual connection
\begin{equation}
\nabla_{-m}^{\ast}:
(\gr_{-m}^W\cV)^{\ast}
\longrightarrow
\omega_X^1 \otimes (\gr_{-m}^W\cV)^{\ast}
\end{equation}
is defined by
\begin{equation}
\nabla_{-m}^{\ast}(\varphi)(v)
=d(\varphi(v))-(\id \otimes \varphi)(\nabla_{-m}(v))
\end{equation}
for $\varphi \in (\gr_{-m}^W\cV)^{\ast}$,
$v \in \gr_{-m}^W\cV$
under the identification
\begin{equation}
\omega_X^1 \otimes (\gr_{-m}^W\cV)^{\ast}
\simeq
\shom_{\cO_X}(\gr_{-m}^W\cV, \omega_X^1),
\end{equation}
the morphism
\begin{equation}
\theta_{m,p}^{\ast}(\chi):
\gr_F^p\gr_m^W\cV^{\ast}
\longrightarrow
\gr_F^{p-1}\gr_m^W\cV^{\ast}
\end{equation}
is given by
\begin{equation}
\theta_{m,p}^{\ast}(\chi)
=-(\theta_{-m,-p+1}(\chi))^{\ast}:
(\gr_F^{-p}\gr_{-m}^W\cV)^{\ast}
\longrightarrow
(\gr_F^{-p+1}\gr_{-m}^W\cV)^{\ast}
\end{equation}
under the identification
$\gr_F^p\gr_m^W\cV^{\ast} \simeq (\gr_F^{-p}\gr_{-m}^W\cV)^{\ast}$ and
$\gr_F^{p-1}\gr_m^W\cV^{\ast} \simeq (\gr_F^{-p+1}\gr_{-m}^W\cV)^{\ast}$
for all $m,p$.
Therefore we have
\begin{equation}
\label{eq:27}
\theta^{\ast}(\chi)=-(\theta(\chi))^{\ast}
\end{equation}
for any $\chi \in \Theta_X(\log D)$
via the identifications above.
\end{para}

In \cite[Definition 4.1]{Brunebarbe},
the category $\gpfmhs(X,D)_{\bR}$
was defined.
The following example shows that
its definition contains some uncertainty.
This is the reason why
we introduce a new category $\fpvhs(X,D)_{\bR}$
in this article.

\begin{exa}
\label{exa:1}
Let $(X, D)$ be a log pair and
$(\cV, F)$ an object of $\fvb(X)$.
The set of all the infinite sequences
$m=(m_1, m_2, m_3, \dots)$ with $m_i \in \bZ$
is denoted by $\bZ^{\bN}$
and the subset of $\bZ^{\bN}$
consisting of $m=(m_1, m_2, m_3, \dots)$
with $m_i=0$ for almost all $i$
is denoted by $\bZ^{\infty}$
(cf. the paragraph before Definition 4.1
in \cite{Brunebarbe}).
For $m=(m_1, m_2, m_3, \dots) \in \bZ^{\infty}$,
we set
\begin{equation}
\begin{cases}
W_m\cV=0 \quad
&\text{if $m_1 \le 0$} \\
W_m\cV=\cV \quad
&\text{if $m_1 \ge 1$}.
\end{cases}
\end{equation}
Then we can easily check
that $W$ defines an increasing filtration on $\cV$.
Moreover, it is easy to see the equality
\begin{equation}
\gr_m^W\cV
=W_m\cV\Bigl/\bigcup_{m' < m}W_{m'}\cV
=0
\end{equation}
for all $m \in \bZ^{\infty}$.
Therefore any object $(\cV, F)$ of $\fvb(X)$
underlies an object of $\gpfmhs(X,D)_{\bR}$
by definition.
\end{exa}

\section{Restriction functor}
\label{sec:restriction-functor}

In this section,
an alternative restriction functor
$\Phi_{D(J)}$ will be constructed
according to the original idea
of Dr. Brunebarbe.

\begin{para}
Let $(X, D)$ be a log pair.
We use the notation in \ref{notn:1}.
For any subset $J \subset I$,
the restriction functor
\begin{equation}
\label{eq:4}
\fvb(X) \longrightarrow \fvb(D(J))
\end{equation}
is defined by assigning the object
$(\cO_{D(J)} \otimes \cV, F)$ of $\fvb(D(J))$
to an object $(\cV, F)$ of $\fvb(X)$,
where $F$ on $\cO_{D(J)} \otimes \cV$
denotes the filtration induced from $F$ on $\cV$
by \eqref{eq:1}.
The functor $\Phi_{D(J)}$ is to be defined
as the restriction of the functor \eqref{eq:4}
to the full subcategory $\fpvhs(X,D)_{\bR}$
of $\fvb(X)$.
Therefore it is sufficient to prove the following:
for an object $(\cV, F)$ of $\fpvhs(X,D)_{\bR}$,
the object $(\cO_{D(J)} \otimes \cV, F)$ of $\fvb(D(J))$
admits a structure of filtered polarized variation of $\bR$-Hodge structures.

From now on,
we use the notation
$\cV(J)=\cO_{D(J)} \otimes \cV$ for short.
\end{para}

\begin{para}
\label{para:3}
Let $(\cV, F)$ be an object of $\fpvhs(X,D)_{\bR}$.
First we treat the pure case.
Namely, we are given the data
$(\nabla, \bV, S)$ consisting of
\begin{itemize}
\item
a nilpotent integrable log connection
$\nabla$ on $\cV$
satisfying the Griffiths transversality,
\item
an $\bR$-local subsystem $\bV$ of $\kernel(\nabla)|_{X \setminus D}$
with the property
$\bC \otimes \bV=\kernel(\nabla)|_{X \setminus D}$,
\item
a morphism of $\cO_X$-modules
$S: \cV \otimes \cV \longrightarrow \cO_X$
satisfying the condition as in \eqref{eq:6}
\end{itemize}
such that
$\gr_F^p\cV$ is locally free of finite rank for all $p$
and that the data
\begin{equation}
(\bV, (\cV, F)|_{X \setminus D},
\nabla|_{X \setminus D}, S|_{X \setminus D})
\end{equation}
is a variation of polarized $\bR$-Hodge structures
of certain weight on $X \setminus D$.

Let $J$ be a subset of $I$.
The finite increasing filtration $W(J)$ on $\cV(J)$
is already defined in \ref{para:1}
as the case of $K=J$.
Moreover, we obtain the following data
by applying the construction
in \cite[Section 5]{Fujino-Fujisawa}:
\begin{itemize}
\item
a nilpotent integrable log connection
$\nabla_k(J)$ on $\gr_k^{W(J)}\cV(J)$
satisfying the Griffiths transversality,
\item
a morphism of $\cO_{D(J)}$-modules
\begin{equation}
S_k(J):
\gr_k^{W(J)}\cV(J)
\otimes
\gr_k^{W(J)}\cV(J)
\longrightarrow
\cO_{D(J)}
\end{equation}
satisfying the condition as in \eqref{eq:6}.
\end{itemize}
However, the construction of the real structure
is missing in \cite[Section 5]{Fujino-Fujisawa}.
Here we present how we can obtain the real structure.
\end{para}

\begin{para}
The $\bR$-structure will be constructed
by gluing the local data.
So we return to the local situation.

Let $(X, D)$ be as in \ref{para:2}
in addition to the situation above.
Then we may assume
$\cV=\cO_X \otimes V$
and that $\nabla$ and $\bV$
are described by
\eqref{eq:2} and \eqref{eq:5}
as in \ref{para:2}.
\end{para}

\begin{lem}
\label{lem:1}
In the situation above,
\begin{equation}
N_j(W(J)_k) \subset W(J)_{k-2}
\end{equation}
for all $k$
and for any $j \in J$.
\end{lem}
\begin{proof}
By \cite[(3.3) Theorem]{Cattani-Kaplan},
$W(J)$ is the monodromy weight filtration
for the nilpotent endomorphism
$\sum_{i \in J}c_iN_i$ for any $c_i > 0$.
Therefore we have
\begin{equation}
\Bigl(N_j+c(\sum_{i \in J \setminus \{j\}}N_i)\Bigr)(W(J)_k)
\subset
W(J)_{k-2}
\end{equation}
for all $k$ and for any $c > 0$.
Thus we obtain the conclusion by sending $c \searrow 0$.
\end{proof}

\begin{para}
\label{para:4}
Now we take other coordinates $(s_1, s_2, \dots, s_n)$
of $X=\Delta^n$
with $D=\{s_1s_2 \cdots s_l=0\}$.
We may assume that $t_i$ and $s_i$ define the same divisor on $X$
for $i=1,2, \dots, l$.
Therefore there exist nowhere vanishing holomorphic functions $a_i$ on $X$
such that $s_i=a_it_i$ for $i=1,2, \dots, l$.

On the other hand,
the coordinates $(s_1, s_2, \dots, s_n)$ induce an isomorphism
$\psi: \cO_X \otimes V\longrightarrow \cV$
such that
$\psi^{\ast}\nabla$ and $\psi^{-1}\bV$
are described
as \eqref{eq:2} and \eqref{eq:5}
by using $s_i$ instead of $t_i$.
Then we can easily see that
the isomorphism
$\psi^{-1} \cdot \varphi : \cO_X \otimes V \longrightarrow \cO_X \otimes V$
is given by
\begin{equation}
\exp((2\pi\sqrt{-1})^{-1}\sum_{i=1}^l(\log a_i)N_i):
\cO_X \otimes V \longrightarrow \cO_X \otimes V
\end{equation}
by choosing an appropriate branch of $\log a_i$ for $i=1,2, \dots, l$.
By restricting to $D(J)$,
we obtain $\cV(J)$ and
\begin{equation}
\id \otimes \varphi:
\cO_{D(J)} \otimes V \longrightarrow \cV(J), \quad
\id \otimes \psi:
\cO_{D(J)} \otimes V \longrightarrow \cV(J)
\end{equation}
such that the isomorphism
$(\id \otimes \psi)^{-1} \cdot (\id \otimes \varphi)$
is given by
\begin{equation}
\label{eq:9}
\exp((2\pi\sqrt{-1})^{-1}\sum_{i=1}^l(\log a_i)N_i):
\cO_{D(J)} \otimes V \longrightarrow \cO_{D(J)} \otimes V.
\end{equation}
Thus the $\bR$-structures $\bV(J)$ on $\cO_{D(J)} \otimes V$
defined by using the coordinates $(t_1, t_2, \dots, t_n)$
and by using $(s_1,s_2, \dots, s_n)$
as in \eqref{eq:8}
are not identified via this isomorphism.
However, the isomorphism
$\gr_k^{W(J)}((\id \otimes \psi)^{-1} \cdot (\id \otimes \varphi))$
of the $\cO_{D(J)}$-modules
$\gr_k^{W(J)}(\cO_{D(J)} \otimes V)=\cO_{D(J)} \otimes \gr_k^{W(J)}V$
is given by
\begin{equation}
\exp((2\pi\sqrt{-1})^{-1}\sum_{i \in I \setminus J}(\log a_i)N_i):
\cO_{D(J)} \otimes \gr_k^{W(J)}V
\longrightarrow
\cO_{D(J)} \otimes \gr_k^{W(J)}V
\end{equation}
by Lemma \ref{lem:1}.
Therefore the $\bR$-structures
$\gr_k^{W(J)}\bV(J)$ on $\cO_{D(J)} \otimes \gr_k^{W(J)}V$
defined by using the coordinates $(t_1, t_2, \dots, t_n)$
and by using $(s_1,s_2, \dots, s_n)$
coincide via the isomorphism
$\gr_k^{W(J)}((\id \otimes \psi)^{-1} \cdot (\id \otimes \varphi))$.
Thus we obtain a globally defined $\bR$-structure
on $\gr_k^{W(J)}(\cO_{D(J)} \otimes V)$,
which is denoted by $\bV_k(J)$.
Hence we obtain the following lemma.
\end{para}

\begin{lem}
\label{lem:7}
Let $(X, D)$ be a log pair
and $(\cV, F)$ an object of $\fpvhs(X,D)_{\bR}$.
Moreover, we assume that $(\cV, F)$ is pure
as in {\rm \ref{para:3}}.
For any subset $J \subset I$,
the data
\begin{equation}
(W(J), \{\nabla_k(J)\}, \{\bV_k(J)\}, \{S_k(J)\})
\end{equation}
given in {\rm \ref{para:3}} and {\rm \ref{para:4}}
is a structure of filtered polarized variation of $\bR$-Hodge structure
on $(\cV(J), F)$.
Thus 
$(\cV(J), F)$
is an object of $\fpvhs(D(J), D(J) \cap D_{I \setminus J})_{\bR}$
for all $J \subset I$.
\end{lem}
\begin{proof}
See Corollary 5.13 and Proposition 5.19 in \cite{Fujino-Fujisawa}.
\end{proof}

\begin{para}
\label{para:5}
Next, we treat the general case.
Let $(X,D)$ be a log pair
and $(\cV, F)$ an object of $\fpvhs(X,D)_{\bR}$.
We fix an structure of filtered polarized variation of $\bR$-Hodge structure
\begin{equation}
(W, \{\nabla_m\}, \{\bV_m\}, \{S_m\})
\end{equation}
on $(\cV, F)$.

Let $J$ be a subset of $I$.
We apply the construction above to 
$(\gr_m^W\cV, F)$,
and obtain
a structure of filtered polarized variation of $\bR$-Hodge structure
\begin{equation}
(W(J), \{\bV_l(J)\}, \{\nabla_l(J)\}, \{S_l(J)\})
\end{equation}
on $(\cO_{D(J)} \otimes \gr_m^W\cV,F)$.
By Lemma \ref{lem:5},
we have the canonical isomorphism
\begin{equation}
\label{eq:19}
\cO_{D(J)} \otimes \gr_m^W\cV
\overset{\simeq}{\longrightarrow}
\gr_m^W\cV(J),
\end{equation}
under which the filtration $F$ on the both sides coincide.
Introducing the data
$W(J), \nabla_l, \bV_l, S_l$ on $\gr_m^W\cV(J)$
from the data $W(J), \nabla_l(J), \bV_l(J), S_l(J)$
on $\cO_{D(J)} \otimes \gr_m^W\cV$
via the identification \eqref{eq:19},
we obtain
a structure of filtered polarized variation of $\bR$-Hodge structure
on $(\gr_m^W\cV(J),F)$.
The canonical surjection
\begin{equation}
W_m(\cV(J))
\longrightarrow
\gr_m^W\cV(J)
\end{equation}
is denoted by $\pi_m$.
Then we have
\begin{equation}
\begin{split}
W_{m-1}(\cV(J))
&\subset
\pi_m^{-1}(W(J)_{l-1}(\gr_m^W\cV(J))) \\
&\subset
\pi_m^{-1}(W(J)_l(\gr_m^W\cV(J)))
\subset
W_m(\cV(J))
\end{split}
\end{equation}
for all $m,l$.
Since $W$ on $\cV$
and $W(J)$ on $\gr_m^W\cV(J)$
are finite filtrations,
we obtain a refinement $(M, \varphi)$ of $W$ on $\cV(J)$
satisfying the following properties:
\begin{itemize}
\item
For any $k \in \bZ$, there exists an integer $l$
such that
\begin{equation}
\begin{split}
&M_k(\cV(J))
=\pi_{m(k)}^{-1}(W(J)_l(\gr_{m(k)}^W\cV(J))), \\
&M_{k-1}(\cV(J))
=\pi_{m(k)}^{-1}(W(J)_{l-1}(\gr_{m(k)}^W\cV(J))),
\end{split}
\end{equation}
\end{itemize}
where $m(k)$ is the integer defined in Definition \ref{defn:1}
for $\varphi$.
We fix the integer $l$
satisfying the conditions above for $k$
and denote it by $l(k)$.
Then we have
\begin{equation}
\begin{split}
&M_k\gr_{m(k)}^W\cV(J)=W(J)_{l(k)}\gr_{m(k)}^W\cV(J), \\
&M_{k-1}\gr_{m(k)}^W\cV(J)=W(J)_{l(k)-1}\gr_{m(k)}^W\cV(J),
\end{split}
\end{equation}
which implies
\begin{equation}
\label{eq:23}
\gr_k^M\gr_{m(k)}^W\cV(J)=\gr_{l(k)}^{W(J)}\gr_{m(k)}^W\cV(J)
\end{equation}
for any $k$.
By Lemma \ref{lem:4} and Lemma \ref{lem:2},
we have the isomorphism
\begin{equation}
\label{eq:20}
\gr_k^M\cV(J)
\simeq
\gr_k^M\gr_{m(k)}^W\cV(J)
\end{equation}
under which the filtration $F$ on the both sides coincide.
Then the data
$\nabla_k, \bV_k, S_k$
on $\gr_k^M\cV(J)$
can be defined from the ones
$\nabla_{l(k)}, \bV_{l(k)}, S_{l(k)}$
on $\gr_{l(k)}^{W(J)}\gr_{m(k)}^W\cV(J)$
via \eqref{eq:23} and \eqref{eq:20}.
It is trivial that
the data
\begin{equation}
(M, \{\nabla_k\}, \{\bV_k\}, \{S_k\})
\end{equation}
is a structure of filtered polarized variation of $\bR$-Hodge structure
on $(\cV(J), F)$.
Thus we conclude the following:
\end{para}

\begin{lem}
Let $(X,D)$ be a log pair,
and $(\cV, F)$ an object of $\fpvhs(X,D)_{\bR}$.
For any $J \subset I$,
the restriction $(\cO_{D(J)} \otimes \cV, F)$
admits a structure of filtered polarized variation of $\bR$-Hodge structure.
In other words,
$(\cO_{D(J)} \otimes \cV, F)$ is an object
of $\fpvhs(D(J), D(J) \cap D_{I \setminus J})_{\bR}$.
\end{lem}

\begin{defn}
The restriction of the functor \eqref{eq:4}
to the full subcategory $\fpvhs(X,D)_{\bR}$ of $\fvb(X)$
gives us a functor
\begin{equation}
\fpvhs(X, D)_{\bR}
\longrightarrow
\fpvhs(D(J), D(J) \cap D_{I \setminus J})_{\bR}
\end{equation}
for any $J \subset I$.
This functor is denoted by $\Phi_{D(J)}$.
\end{defn}

\begin{rmk}
Let $(\cV, F)$ be an object of
$\fpvhs(X,D)_{\bR}$.
Since $\gr_F^p\cV$ is locally free of finite rank for all $p$
as mentioned in Remark \ref{rmk:1},
the canonical morphism
\begin{equation}
\label{eq:29}
\cO_{D(J)} \otimes \gr_F^p\cV
\longrightarrow
\gr_F^p\cV(J)
\end{equation}
is an isomorphism for all $p$.
\end{rmk}

\begin{para}
In the remainder of this section,
we study the Higgs field of $\Phi_{D(J)}(\cV, F)$
for an object $(\cV,F)$ of $\fpvhs(X,D)_{\bR}$
for the later use.

Let $(X, D)$ be a log pair.
Here we recall that the canonical morphism
$\Omega^1_X(\log D_{I \setminus J})
\longrightarrow \omega^1_X=\Omega^1_X(\log D)$
induces the morphism
\begin{equation}
\label{eq:33}
\omega^1_{D(J)}
\longrightarrow
\omega^1_X \otimes \cO_{D(J)},
\end{equation}
which fits in the exact sequence
\begin{equation}
\label{eq:40}
\begin{CD}
0 @>>> \omega_{D(J)}^1
  @>>> \omega_X^1 \otimes \cO_{D(J)}
  @>>> \cO_{D(J)}^{|J|}
  @>>> 0
\end{CD}
\end{equation}
as in \cite[5.14]{Fujino-Fujisawa}.
\end{para}

\begin{para}
Let $(\cV,F)$ be a pure object of $\fpvhs(X,D)_{\bR}$
and $(\nabla, \bV, S)$ the data on $(\cV,F)$ as in \ref{para:3}.
Then the structure of filtered polarized variation of $\bR$-Hodge structure
\begin{equation}
(W(J), \{\nabla_k(J)\}, \{\bV_k(J)\}, \{S_k(J)\}),
\end{equation}
on $(\cV(J), F)=\Phi_{D(J)}(\cV,F)$
is constructed in \ref{para:3} and \ref{para:4}.
The Higgs field
\begin{equation}
\gr_F^p\gr_k^{W(J)}\cV(J)
\longrightarrow
\omega^1_{D(J)} \otimes \gr_F^{p-1}\gr_k^{W(J)}\cV(J)
\end{equation}
associated to $\nabla_k(J)$
is denoted by $\theta(J)_{k,p}$ for all $k,p$.

On the other hand,
the Higgs field $\theta_p$ associated to $\nabla$ on $\cV$
induces the morphism
\begin{equation}
\id \otimes \theta_p:
\cO_{D(J)} \otimes \gr_F^p\cV
\longrightarrow
\omega^1_X \otimes (\cO_{D(J)} \otimes \gr_F^{p-1}\cV)
\end{equation}
by the identification
$\cO_{D(J)} \otimes \omega^1_X \otimes \gr_F^{p-1}\cV \simeq
\omega^1_X \otimes (\cO_{D(J)} \otimes \gr_F^{p-1}\cV)$.
Via the canonical isomorphism
\eqref{eq:29},
we obtain a morphism of $\cO_{D(J)}$-modules
\begin{equation}
\label{eq:34}
\gr_F^p\cV(J)
\longrightarrow
\omega_X^1 \otimes \gr_F^{p-1}\cV(J)
\end{equation}
for all $p$.
\end{para}

\begin{lem}
\label{lem:8}
In the situation above,
the morphism \eqref{eq:34}
preserves the filtration $W(J)$ on the both sides.
Therefore it induces the morphism
\begin{equation}
\gr_k^{W(J)}\gr_F^p\cV(J)
\longrightarrow
\omega^1_X \otimes \gr_k^{W(J)}\gr_F^{p-1}\cV(J),
\end{equation}
which fits in the commutative diagram
\begin{equation}
\xymatrix{
\gr_F^p\gr_k^{W(J)}\cV(J)
\ar[r]
\ar[d]^{\theta(J)_{k,p}}
&
\gr_k^{W(J)}\gr_F^p\cV(J)
\ar[dd] \\
\omega^1_{D(J)} \otimes \gr_F^{p-1}\gr_k^{W(J)}\cV(J) \ar[d] & \\
\omega^1_X \otimes \gr_F^{p-1}\gr_k^{W(J)}\cV(J)
\ar[r]
& \omega^1_X \otimes \gr_k^{W(J)}\gr_F^{p-1}\cV(J)
}
\end{equation}
where the two horizontal arrows
are the canonical isomorphism
for switching the filtrations $F$ and $W(J)$
and the bottom left vertical arrow
is induced by the canonical inclusion
\eqref{eq:33}.
\end{lem}
\begin{proof}
The connection $\nabla$ induces a $\bC$-morphism
\begin{equation}
\cO_{D(J)} \otimes \cV
\longrightarrow
\omega^1_X \otimes \cO_{D(J)} \otimes \cV.
\end{equation}
By the local description in \ref{para:2},
this morphism preserves the filtration $W(J)$.
Thus we obtain the conclusion
by the definition of $\nabla_k(J)$
in \cite[5.14]{Fujino-Fujisawa}.
\end{proof}

\begin{para}
Let $(\cV,F)$ be an object of $\fpvhs(X,D)_{\bR}$
which is not necessarily pure.
We fix
a structure of filtered polarized variation of $\bR$-Hodge structure
\begin{equation}
(W, \{\nabla_m\}, \{\bV_m\}, \{S_m\})
\end{equation}
on $(\cV,F)$.
Then we have
the structure of filtered polarized variation of $\bR$-Hodge structure
\begin{equation}
(M, \{\nabla_k\}, \{\bV_k\}, \{S_k\})
\end{equation}
on $(\cV(J), F)=\Phi_{D(J)}(\cV,F)$
constructed in \ref{para:5}.
The Higgs field
\begin{equation}
\gr_F^p\gr_k^M\cV(J)
\longrightarrow
\omega^1_{D(J)} \otimes \gr_F^{p-1}\gr_k^M\cV(J)
\end{equation}
associated to $\nabla_k$
is denoted by $\theta_{k,p}$ for all $k,p$.
On the other hand,
the Higgs field
associated to $\nabla_m$ on $\gr_m^W\cV$
induces a morphism
\begin{equation}
\label{eq:37}
\gr_F^p\gr_m^W\cV(J)
\longrightarrow
\omega^1_X \otimes \gr_F^{p-1}\gr_m^W\cV(J)
\end{equation}
as before via the canonical isomorphism
\begin{equation}
\cO_{D(J)} \otimes \gr_F^p\gr_m^W\cV
\overset{\simeq}{\longrightarrow}
\gr_F^p\gr_m^W\cV(J)
\end{equation}
for all $m, p$.
\end{para}

\begin{lem}
\label{lem:9}
In the situation above,
the morphism
\eqref{eq:37}
preserves the filtration $M$.
Therefore it induces the morphism
\begin{equation}
\label{eq:38}
\gr_k^M\gr_F^p\gr_{m(k)}^W\cV(J)
\longrightarrow
\omega^1_X
\otimes
\gr_k^M\gr_F^{p-1}\gr_{m(k)}^W\cV(J)
\end{equation}
which fits in the commutative diagram
\begin{equation}
\xymatrix{
\gr_F^p\gr_k^M\cV(J) \ar[d]_{\theta_{k,p}} \ar[r]^-{(4)}
&
\gr_F^p\gr_k^M\gr_{m(k)}^W\cV(J) \ar[d]^{(5)} \\
\omega^1_{D(J)} \otimes \gr_F^{p-1}\gr_k^M\cV(J) \ar[d]_{(1)}
&
\gr_k^M\gr_F^p\gr_{m(k)}^W\cV(J) \ar[d]^{(6)} \\
\omega^1_X \otimes \gr_F^{p-1}\gr_k^M\cV(J) \ar[d]_{(2)}
&
\omega^1_X
\otimes
\gr_k^M\gr_F^{p-1}\gr_{m(k)}^W\cV(J) \ar[d]^{(7)} \\
\omega^1_X \otimes \gr_k^M\gr_F^{p-1}\cV(J) \ar[r]_-{(3)}
&
\omega^1_X
\otimes
\gr_k^M\gr_{m(k)}^W\gr_F^{p-1}\cV(J)
},
\end{equation}
where $(6)$ is the morphism \eqref{eq:38},
$(1)$ is induced by the morphism \eqref{eq:33},
$(2), (5)$ are induced by switching the filtrations $F$ and $M$,
$(3)$ and $(4)$ are induced by the morphism
\eqref{eq:11} for $\gr_F^{p-1}\cV(J)$ and $\cV(J)$,
and $(7)$ is induced from the isomorphism \eqref{eq:16} respectively.
\end{lem}
\begin{proof}
Because of
$M_k\gr_{m(k)}^W\cV(J)=W(J)_{l(k)}\gr_{m(k)}^W\cV(J)$,
Lemma \ref{lem:8} for $\gr_{m(k)}^W\cV$
and Corollary \ref{cor:1}
imply the conclusion.
\end{proof}

\section{Semipositivity theorem}
\label{sec:semip-theor}

First, we recall the definition of
semipositive locally free sheaves.

\begin{defn}
A locally free sheaf $\cE$ of finite rank
on a complete algebraic variety $X$
is said to be semipositive
if $\cO_{\bP_X(\cE)}(1)$ is nef on 
$\bP_{X}(\cE)$. 
\end{defn}

By using the functor $\Phi_{D(J)}$
in Section \ref{sec:restriction-functor},
the semipositivity theorem of Fujita--Zucker--Kawamata
can be generalized
as follows.

\begin{thm}
\label{thm:2}
Let $(X, D)$ be a log pair with $X$ being complete,
$(\cV, F)$ an object of $\fpvhs(X,D)_{\bR}$
and $A$ a locally free $\cO_X$-module of finite rank
equipped with the surjection $\pi: \gr_F\cV \longrightarrow A$.
Assume that there exists
a structure of filtered polarized variation of $\bR$-Hodge structure
$(W, \{\nabla_m\}, \{\bV_m\}, \{S_m\})$
such that
the composite
\begin{equation}
\begin{CD}
\gr_F\gr^W\cV
@>>>
\omega_X^1 \otimes \gr_F\gr^W\cV \\
@>{\simeq}>>
\omega_X^1 \otimes \gr^W\gr_F\cV \\
@>{\id \otimes \gr^W\pi}>>
\omega_X^1 \otimes \gr^WA
\end{CD}
\end{equation}
is the zero morphism,
where the filtration $W$ on $A$ is induced
from the filtration $W$ on $\gr_F\cV$
and where the first arrow is the associated Higgs field,
and the second arrow is the canonical isomorphism
induced by switching the filtrations $W$ and $F$.
Then $A$ is semipositive.
\end{thm}
\begin{proof}
We fix
a structure of filtered polarized variation of $\bR$-Hodge structure
\begin{equation}
(W, \{\nabla_m\}, \{\bV_m\}, \{S_m\})
\end{equation}
satisfying the assumption.

For the case of $\dim X=1$,
we can easily reduce the problem to the pure case.
Then the equality \eqref{eq:27} implies that
$A^{\ast} \subset (\gr_F\cV)^{\ast} \simeq \gr_F\cV^{\ast}$
is contained in the kernel of the Higgs field
$\theta^{\ast}$ of $(\cV^{\ast},F)$.
Therefore we can obtain the conclusion
for $X$ by an argument similar to
the proof of Lemma 4.7 in \cite{Brunebarbe}.
Then we complete the proof by Lemma \ref{lem:11} below
and by the inductive argument
as in
\cite{KawamataCAV},
\cite{Fujino-Fujisawa},
\cite{Brunebarbe}
using the functor $\Phi_{D(J)}$,
together with Lemma \ref{lem:12} below.
\end{proof}

\begin{lem}
\label{lem:11}
In addition to the situation above,
let $(Y,E)$ be a log pair and
$f: (Y,E) \longrightarrow (X,D)$ a  morphism of log pairs.
We have the surjective morphism
\begin{equation}
\label{eq:41}
\gr_Ff^{\ast}\cV \longrightarrow f^{\ast}A
\end{equation}
induced by the canonical isomorphism \eqref{eq:36}.
We fix
the structure of filtered polarized variation of $\bR$-Hodge structure
\eqref{eq:31}.
The filtration $W$ on $A$
is induced from the filtration $W$ on $f^{\ast}\cV$
via the surjection \eqref{eq:41}.
Then the composite
\begin{equation}
\begin{split}
\gr_F\gr^Wf^{\ast}\cV
&\longrightarrow
\omega^1_Y \otimes \gr_F\gr^Wf^{\ast}\cV \\
&\longrightarrow
\omega^1_Y \otimes \gr^W\gr_Ff^{\ast}\cV \\
&\longrightarrow
\omega^1_Y \otimes \gr^Wf^{\ast}A
\end{split}
\end{equation}
is the zero morphism,
where the top arrow is the Higgs field associated to
\eqref{eq:31},
the second is the canonical isomorphism
for switching the filtration $F$ and $W$
and the third is induced from the surjection \eqref{eq:41}.
\end{lem}
\begin{proof}
By the commutative diagram
\begin{equation}
\begin{CD}
f^{\ast}W_m\gr_F\cV @>{\simeq}>>
W_mf^{\ast}\gr_F\cV @>{\simeq}>> W_m\gr_Ff^{\ast}\cV \\
@VVV @VVV @VVV \\
f^{\ast}W_mA @>>> f^{\ast}A @= f^{\ast}A
\end{CD}
\end{equation}
and by the fact that the left vertical arrow is surjective,
we have the canonical morphism
\begin{equation}
f^{\ast}W_mA \longrightarrow W_mf^{\ast}A
\end{equation}
for every $m$.
Here we note that
the filtration $W$ on $f^{\ast}A$ in the right hand side
is induced from the filtration $W$ on $\gr_Ff^{\ast}\cV$.
Then we have the commutative diagram
\begin{equation}
\begin{CD}
f^{\ast}\gr^W\gr_F\cV @>>> \gr^W\gr_Ff^{\ast}\cV \\
@VVV @VVV \\
f^{\ast}\gr^WA @>>> \gr^Wf^{\ast}A ,
\end{CD}
\end{equation}
from which we can easily obtain the conclusion.
\end{proof}

\begin{lem}
\label{lem:12}
In the situation above,
we fix
the structure of filtered polarized variation of $\bR$-Hodge structure
\begin{equation}
\label{eq:39}
(M, \{\nabla_k\},\{\bV_k\},S_k)
\end{equation}
on $(\cV(J),F)=\Phi_{D(J)}(\cV,F)$
constructed in {\rm \ref{para:5}}.
By using the canonical isomorphism \eqref{eq:29}
we obtain a surjective morphism
\begin{equation}
\label{eq:35}
\gr_F\cV(J)
\longrightarrow
\cO_{D(J)} \otimes A
\end{equation}
which fits in the commutative diagram
\begin{equation}
\begin{CD}
\cO_{D(J)} \otimes \gr_F\cV
@>>>
\gr_F\cV(J) \\
@VVV @VVV \\
\cO_{D(J)} \otimes A
@=
\cO_{D(J)} \otimes A
\end{CD}
\end{equation}
by definition.
A filtration $M$ on $\cO_{D(J)} \otimes A$
is induced from $M$ on $\gr_F\cV(J)$ via the surjection \eqref{eq:35}.
Then the composite
\begin{equation}
\begin{split}
\gr_F\gr^M\cV(J)
&\longrightarrow
\omega^1_{D(J)} \otimes \gr_F\gr^M\cV(J) \\
&\longrightarrow
\omega^1_{D(J)} \otimes \gr^M\gr_F\cV(J) \\
&\longrightarrow
\omega^1_{D(J)} \otimes \gr^M(\cO_{D(J)} \otimes A)
\end{split}
\end{equation}
is the zero morphism,
where the first arrow denotes the Higgs field
associated to \eqref{eq:39}
and the second arrow is the isomorphism
induced by switching the filtrations $M$ and $F$.
\end{lem}
\begin{proof}
From the exact sequence
\eqref{eq:40},
the morphism
\begin{equation}
\omega^1_{D(J)} \otimes \gr_k^M(\cO_{D(J)} \otimes A)
\longrightarrow
\omega^1_X \otimes \gr_k^M(\cO_{D(J)} \otimes A)
\end{equation}
is injective.
Therefore it suffices to prove that
the composite
\begin{equation}
\begin{split}
\gr_F^p\gr_k^M\cV(J)
&\xrightarrow{\ \theta_{k,p}\ }
\omega^1_{D(J)} \otimes \gr_F^{p-1}\gr_k^M\cV(J) \\
&\xrightarrow{\ \phantom{\theta_{k,p}}\ }
\omega^1_{D(J)} \otimes \gr_k^M\gr_F^{p-1}\cV(J) \\
&\xrightarrow{\ \phantom{\theta_{k,p}}\ }
\omega^1_{D(J)} \otimes \gr_k^M(\cO_{D(J)} \otimes A) \\
&\xrightarrow{\ \phantom{\theta_{k,p}}\ }
\omega^1_X \otimes \gr_k^M(\cO_{D(J)} \otimes A)
\end{split}
\end{equation}
is the zero morphism for all $k,p$.
Then, by the commutative diagram
\begin{equation}
\begin{CD}
\omega^1_{D(J)} \otimes \gr_F^{p-1}\gr_k^M\cV(J)
@>>>
\omega^1_X \otimes \gr_F^{p-1}\gr_k^M\cV(J) \\
@V{\simeq}VV @VV{\simeq}V \\
\omega^1_{D(J)} \otimes \gr_k^M\gr_F^{p-1}\cV(J)
@>>>
\omega^1_X \otimes \gr_k^M\gr_F^{p-1}\cV(J) \\
@VVV @VVV \\
\omega^1_{D(J)} \otimes \gr_k^M(\cO_{D(J)} \otimes A)
@>>>
\omega^1_X \otimes \gr_k^M(\cO_{D(J)} \otimes A)
\end{CD}
\end{equation}
and by Lemma \ref{lem:9}
it suffices to prove that the composite
\begin{equation}
\begin{split}
\gr_k^M\gr_F^p\gr_{m(k)}^W\cV(J)
&\longrightarrow
\omega^1_X \otimes \gr_k^M\gr_F^{p-1}\gr_{m(k)}^W\cV(J) \\
&\overset{\simeq}{\longrightarrow}
\omega^1_X \otimes \gr_k^M\gr_{m(k)}^W\gr_F^{p-1}\cV(J) \\
&\longrightarrow
\omega^1_X \otimes \gr_k^M\gr_{m(k)}^W(\cO_{D(J)} \otimes A)
\end{split}
\end{equation}
is the zero morphism for all $k,p$,
where the first arrow
is the morphism \eqref{eq:38}
and the second is the one
induced by the isomorphism \eqref{eq:16}.
Here we remark that
the filtration $W$ on $\cO_{D(J)} \otimes A$
is defined as the filtration induced from $W$ on $\gr_F\cV(J)$
via the surjection \eqref{eq:35}.
Therefore we obtain the conclusion
from the commutative diagram
\begin{equation}
\begin{CD}
\cO_{D(J)} \otimes \gr_{m(k)}^W\gr_F^{p-1}\cV
@>>>
\gr_{m(k)}^W\gr_F^{p-1}\cV(J) \\
@VVV @VVV \\
\cO_{D(J)} \otimes \gr_{m(k)}^WA
@>>>
\gr_{m(k)}^W(\cO_{D(J)} \otimes A)
\end{CD}
\end{equation}
and the assumption for $A$.
\end{proof}

Now Theorem 1.8 of \cite{Brunebarbe}
is a corollary of Theorem \ref{thm:2}.

\begin{cor}
\label{cor:2}
Let $(X,D)$ be a log pair
and $(\cV,F)$ an object of $\fpvhs(X,D)_{\bR}$.
We assume that $(\cV,F)$ is pure,
that is, there exists a data $(\nabla,\bV,S)$ as in \ref{para:3}.
Let $A$ be a subbundle of $\gr_F\cV$
contained in the kernel of the associated Higgs field
$\theta: \gr_F\cV \longrightarrow \omega^1_X \otimes \gr_F\cV$.
Then the $\cO_X$-dual $A^{\ast}$ of $A$ is semipositive.
\end{cor}
\begin{proof}
Because of the equality \eqref{eq:27},
the quotient bundle $A^{\ast}$
of $(\gr_F\cV)^{\ast} \simeq \gr_F\cV^{\ast}$
satisfies the assumption in Theorem \ref{thm:2}.
\end{proof}

In \cite[Theorem 4.5]{Brunebarbe},
a subbundle of $\gr_F\cV$ is considered
instead of a quotient bundle of $\gr_F\cV$
in Theorem \ref{thm:2}.
Apparently, it looks possible to obtain the ``semi-negativity''
for a certain kind of subbundles
by using the inductive argument
as in the proof of Theorem \ref{thm:2}.
However, the following example shows that
Theorem 4.5 in \cite{Brunebarbe} is false.

\begin{exa}
\label{exa:2}
Let $U$ be a Zariski open subset of $C=\bP^1$
and $V=(\bV, F)$ be
a polarizable variation of $\bR$-Hodge structure of weight $w$
on $U$ of unipotent monodromy,
where $F$ denotes the Hodge filtration
on $\cO_U \otimes \bV$.
The canonical extension of $\cO_U \otimes \bV$ to the whole $C$
is denoted by $\cV$.
By Schmid's theorem,
the Hodge filtration $F$ extends to $\cV$
such that $\gr_F^p\cV$ is a locally free $\cO_C$-module of finite rank
for all $p$.
Here we assume the following conditions:
\begin{itemize}
\item
For an integer $b$,
$F^{b+1}\cV=0$ and $F^b\cV \simeq \cO_C(n)$
for a positive integer $n$.
\end{itemize}
Now we set $X=\bP_C(\cO_C \oplus \cO_C(n))$.
The projection $X \longrightarrow C$
is denoted by $\pi: X \longrightarrow C$.
The minimal section is denoted by $C_0$,
that is $C_0$ is the section of $\pi$ with $C_0^2=-n$.
We denote by $C_{\infty}$ the section of $\pi$
with the property $C_{\infty}=n$.
Then $C_0 \cap C_{\infty}=\emptyset$ and
$\cO_X(C_{\infty}) \simeq \cO_X(C_0) \otimes \pi^{\ast}\cO_C(n)$.

The local system $\pi^{-1}\bV$
underlies a polarizable variation of Hodge structure
on $\pi^{-1}U \subset X$
such that its canonical extension is
$\pi^{\ast}\cV$ with the filtration
$F^p\pi^{\ast}\cV=\pi^{\ast}F^p\cV$ for all $p$.

Now we define a locally free $\cO_X$-module $\widetilde{\cV}$ of finite rank
equipped with an increasing filtration $W$
and a decreasing filtration $F$ by
\begin{equation}
\begin{split}
&\widetilde{\cV}=\pi^{\ast}\cV \oplus \cO_X \\
&W_{-1}\widetilde{\cV}=0, \quad
W_0\widetilde{\cV}=\pi^{\ast}\cV, \quad
W_1\widetilde{\cV}=\widetilde{\cV} \\
&F^p\widetilde{\cV}=F^p\pi^{\ast}\cV \oplus \cO_X \ (p \le b), \quad
F^{b+1}\widetilde{\cV}=0.
\end{split}
\end{equation}
Then
\begin{equation}
\begin{split}
&\gr_m^W\widetilde{\cV}=0 \quad \text{if $m \notin \{0, 1\}$} \\
&\gr_0^W\widetilde{\cV}=\pi^{\ast}\cV \\
&\gr_1^W\widetilde{\cV}=\cO_X
\end{split}
\end{equation}
by definition.
Thus $(\widetilde{\cV}, W, F)$ underlies an object
of $\gpfmhs(X,D)$, where $D=\pi^{-1}(C \setminus U)$.
Here we note
\begin{equation}
F^b\widetilde{\cV}
=F^b\pi^{\ast}\cV \oplus \cO_X
=\pi^{\ast}\cO_C(n) \oplus \cO_X
\simeq
\cO_X(C_{\infty}-C_0) \oplus \cO_X
\end{equation}
by the isomorphism
\begin{equation}
\cO_X(C_{\infty}-C_0) \simeq \pi^{\ast}\cO_C(n)
\end{equation}
above.
On the other hand,
we have the commutative diagram
\begin{equation}
\begin{CD}
0 @>>> \cO_X(-C_0)
  @>>> \cO_X
  @>>> \cO_{C_0} @>>> 0 \\
@. @VVV @VVV @| \\
0 @>>> \cO_X(C_{\infty}-C_0)
  @>>> \cO_X(C_{\infty})
  @>>> \cO_{C_0} @>>> 0
\end{CD}
\end{equation}
with exact rows
because of $C_0 \cap C_{\infty}=\emptyset$.
Then we can easily obtain an exact sequence
\begin{equation}
\begin{CD}
0 @>>> \cO_X(-C_0)
  @>{\varphi}>> \cO_X(C_{\infty}-C_0) \oplus \cO_X
  @>>> \cO_X(C_{\infty}) @>>> 0
\end{CD}
\end{equation}
from the commutative diagram above.
We set
\begin{equation}
A=\image(\varphi)
\subset \cO_X(C_{\infty}-C_0) \oplus \cO_X
\simeq F^b\widetilde{\cV}
=\gr_F^b\widetilde{\cV}
\subset \gr_F\widetilde{\cV},
\end{equation}
which is isomorphic to $\cO_X(-C_0)$.
We have
$\gr_F^b\widetilde{\cV}\bigl/A \simeq \cO_X(C_{\infty})$
by the exact sequence above.
We easily see the equalities
\begin{equation}
W_0\widetilde{\cV} \cap A=0, \quad
W_1\widetilde{\cV} \cap A=A
\end{equation}
by definition.
Therefore we have
\begin{equation}
\begin{split}
&\gr_m^WA=0 \quad \text{for $m \not= 1$} \\
&\gr_1^WA=A \simeq \cO_X(-C_0)
\subset \gr_1^W\gr_F^b\widetilde{\cV}=\cO_X
\end{split}
\end{equation}
where $W$ denotes the filtration on $A$
induced from $W$ on $\widetilde{\cV}$.
Hence $\gr^WA$ is contained in the kernel of the Higgs field
\begin{equation}
\theta:
\gr^W\gr_F\widetilde{\cV}
\longrightarrow
\Omega_X^1(\log D) \otimes \gr^W\gr_F\widetilde{\cV}
\end{equation}
because $\gr_1^W\gr_F^{b-1}\widetilde{\cV}=0$.
Thus $A$ satisfies the assumption
in Theorem 4.5 of \cite{Brunebarbe}.
However, the $\cO_X$-dual $A^{\ast} \simeq \cO_X(C_0)$
is not semipositive
because $C_0^2=-n < 0$
\end{exa}

\providecommand{\bysame}{\leavevmode\hbox to3em{\hrulefill}\thinspace}

\end{document}